\documentclass[12pt]{amsart}

\usepackage{amsmath,amssymb,amsthm,enumitem}

\usepackage{multirow}
\usepackage{tikz}
\usepackage{amsfonts}

\usetikzlibrary{matrix,arrows,decorations.pathmorphing,decorations.pathreplacing}

\usepackage{enumitem}
\usepackage{soul}

\usepackage[justification=centering]{caption}

\usepackage{ulem}

\usepackage{hyperref}

\usepackage{diagbox}

\makeatletter
\@namedef{subjclassname@2010}{%
  \textup{2010} Mathematics Subject Classification}
\makeatother

\DeclareMathOperator{\Gal}{Gal}

\DeclareMathOperator{\Max}{Max}

\renewcommand{\phi}{\varphi}

\usepackage{todonotes}

\newtheorem{theorem}{Theorem}[section]
\newtheorem*{thm}{Theorem}

\newtheorem{proposition}[theorem]{Proposition}
\newtheorem{lemma}[theorem]{Lemma}
\newtheorem{corollary}[theorem]{Corollary}

\theoremstyle{definition}

\def\cqfd{
{\hfill
\kern 6pt\penalty 500
\raise -1pt\hbox{\vrule\vbox to 5pt{\hrule width 4pt
\vfill\hrule}\vrule}}
\break}

\font\tengoth=eufm10
\font\sevengoth=eufm7
\font\fivegoth=eufm5
\newfam\gothfam
\textfont\gothfam=\tengoth\scriptfont\gothfam=\sevengoth\scriptscriptfont\gothfam=\fivegoth

\title[Closed points on  curves over finite fields]{Closed points on  curves over finite fields}

\author[Aubry]{Yves Aubry}
\address[Aubry]{Institut de Math\'ematiques de Toulon - IMATH, Universit\'e de Toulon, France}
\address[Aubry]{Institut de Math\'ematiques de Marseille - I2M, Aix Marseille Univ, CNRS, Centrale Marseille, France}
\email{yves.aubry@univ-tln.fr}

\author[Herbaut]{Fabien Herbaut}
\address[Herbaut]{INSPE Nice-Toulon, Universit\'e C\^ote d'Azur, France}
\address[Herbaut]{Institut de Math\'ematiques de Toulon - IMATH, Universit\'e de Toulon, France}

\email{fabien.herbaut@univ-cotedazur.fr}

\author[Monaldi]{Julien Monaldi}
\address[Monaldi]{Institut de Math\'ematiques de Toulon - IMATH, Universit\'e de Toulon, France}
\email{julien.monaldi@ac-nice.fr}

\begin{document} 

\footnote{This work is partially supported by the French Agence Nationale de la Recherche through the Barracuda project under Contract ANR-21-CE39-0009-BARRACUDA.}


\baselineskip=17pt


\begin{abstract}
We are interested in the quantity $\rho(q,g)$ defined 
as the smallest  positive integer  such that
$r\geq \rho(q,g)$ implies that
any absolutely irreducible smooth projective algebraic curve defined over ${\mathbb F}_q$ of genus $g$
has a closed point of degree $r$.
We provide general upper bounds for this number 
and its exact value 
for $g=1,2$ and $3$.
We also improve the known upper bounds on the number 
of closed points of degree 2 on a curve.

\end{abstract}

\date{\today}



\maketitle

\section{Introduction}
In the whole paper we consider  a power $q$ of a prime,
$\mathbb{F}_q$ the finite field with $q$ elements and
$\overline{\mathbb{F}}_q$ its algebraic closure.
Let $X$ be
an absolutely irreducible smooth projective algebraic curve   (just called curve from now on)
 defined over ${\mathbb F}_q$ of genus $g$.
This article deals with the notion of closed point, that is 
an orbit
under the action
of $\Gal \left( 
\overline{\mathbb{F}}_q /
 \mathbb{F}_q 
\right) $
on $X(\overline{\mathbb{F}}_q)$,
or equivalently 
 a place of the corresponding function field.
If $P$ is a closed point of $X$, we define its degree as the
cardinality of the orbit, 
or equivalently as the dimension over $\mathbb{F}_q $ of the residue
field,
that is the quotient of the local ring 
$\mathcal{O}_{P,X}$ by its maximal ideal $\mathcal{M}_{P,X} $.
Clearly the numbers 
$B_r(X)$ ($B_r$ for short) of closed points of degree $r$
are related to the numbers 
$N_r(X)=\sharp X({\mathbb F}_{q^r})$ 
($N_r$ for short) 
of rational points over ${\mathbb F}_{q^r}$ 
by the following formula:
\begin{equation}\label{link_N_r-B_d}
N_r(X)=\sum_{d\mid r}dB_d(X).
\end{equation}

The aim of the article is to study the  quantity $\rho(q,g)$
introduced in Problem 3.2.13 of \cite{TVN} by
Tsfasman,  Vl\u adu\c t  and  Nogin  and defined as
the smallest
 positive integer  such that
$r\geq \rho(q,g)$ implies that $B_r(X)\geq 1$ 
for any genus $g$ curve $X$ defined over ${\mathbb F}_q$.

They ask to find its exact value. 
In this paper  we contribute to this issue
by determining the exact value of $\rho(q,g)$ for $g=1, 2, 3$. 
We also
establish a new upper bound for the number 
$B_2(X)$  of  closed points of degree 2.
We summarize our contributions in the following theorem.

\begin{thm}(Proposition \ref{Borne_B2},Theorem \ref{theorem_genre_1}, Theorem \ref{theorem_genre_2}, Theorem \ref{theorem_genre_3})
Let $g$ be a positive integer and $X$ be a genus $g$ curve defined over ${\mathbb F}_q$.
\begin{enumerate}[label=(\roman*)]

\item An upper bound on $B_2(X)$ is given by

$$
B_2(X) \leq \left\{
    \begin{array}{ll}
        \frac{q^2+1+2gq}{2}-\frac{(q+1)^2}{2g} & \mbox{if} \ \ g \geq 2q+2 \\
         \frac{q^2+1+2gq}{2}-\frac{4(q+1)-g}{8} & \mbox{otherwise.}
    \end{array}
\right.
$$ \ \\

\item The values of $\rho(q,g)$ for $g=1,2$ and $3$ are as follows
\begin{center}
\begin{table}[!h]
\centering
\footnotesize
\begin{tabular}{|c|ccccccccccccccccccc|}
        \hline 
        \diagbox{$g$}{$q$} & $2$ & $3$ & $4$ & $5$ & $7 $ & $8 $ & $9$ & $11$ & $13$ & $16$ & $17$ & $19$ & $23$ & $25$ & $27 $& $29$ & $31$ & $32$ & $\geq 37$   \\
        \hline    
        $1$ & $5$ & $3$ & $3$ & $1$ & $\cdots$ &  &  & & & & & & & & & &  &$\cdots$  &$1$  \\
        $2$ & $4$ & $4$ & $2$ & $3$ & $2$ &$2$&$2$ & $2$ & $1$ & $\cdots$ & & & & & & &   & $\cdots$ & $1$  \\
        $3$ & $7$ & $5$ & $3$ & $3$ & $2$ & $2$ & $3$ & $2$ & $2$ & $2$ & $2$ & $2$ & $2$ & $2$ & $1$ & $2$ & $1$ & $2$ & $1$   \\
        \hline
\end{tabular} 
\normalsize
\caption{Values of $\rho(q,g)$ for $g=1,2$ and $3$.}  
\end{table}
\end{center}
\end{enumerate}
\end{thm}

Let us stress the connection between the topic
and the notion
of Diophantine stability introduced (over number fields) by Mazur and Rubin in \cite{M-R}: a variety $V$ defined over a field $K$ is said 
to be diophantine-stable for the field extension $L/K$ if $V(K)=V(L)$.
When $r$ is a prime number
and $X$ is a curve 
 defined over ${\mathbb F}_q$
 which is diophantine-stable for the extension 
 ${\mathbb F}_{q^r}/{\mathbb F}_q$ then
Formula (\ref{link_N_r-B_d})  implies $\rho(q,g)>r$.
 Curves with Diophantine stability (also called DS-curves) over finite fields have been studied by  Lario 
 who provides in \cite{Lario}
the complete list of isomorphism classes of DS-curves up to genus 3.
Vrioni has studied in \cite{V}  curves and surfaces with  Diophantine stability over finite fields.

\bigskip

\section{Known bounds on $B_r$ and   $\rho(q,g)$}
We collect in this section the known general bounds on $B_r$ and  $\rho(q,g)$
as well as existence results for points of given degree.
In this direction we also propose one contribution,
namely Proposition
\ref{Borne_avec_A},
which sometimes slighty improve the known bounds
and which prove useful 
to save some cases in the study of the last three sections.

The zeta function $Z_X$ of a curve $X$ of genus $g$ defined over ${\mathbb F}_q$ 
is defined by 
 $$Z_X(T)=\exp\Bigl(\sum_{n=1}^{\infty}\sharp X({\mathbb F}_{q^n})\frac{T^n}{n}\Bigr).$$ 
 It is well-known that it is a rational fraction of the form
 \begin{displaymath}
Z_X(T)=\frac{\prod_{j=1}^g(1-\omega_j T)(1-\overline{\omega_j}T)}{(1-T)(1-qT)}
\end{displaymath}
 and the Riemann Hypothesis, proved by Hasse for elliptic curves and by Weil for curves of any genus,
 says that the $\omega_j$'s are  complex numbers of absolute value $\sqrt{q}$.
  We obtain from the two previous expressions of the zeta function 
 \begin{equation}\label{Nombre_de_points_et_omega_i}
 N_n(X)=\sharp X({\mathbb F}_{q^n})=q^n+1-\sum_{j=1}^g(\omega_j^n+\overline{\omega_j}^n)
 \end{equation}
 for any $n\geq 1$, and 
the Riemann Hypothesis  implies that:
 \begin{equation}\label{Weil_Bounds}
 q^n+1-2gq^{n/2} \leq N_n(X) \leq q^n+1+2gq^{n/2}.
 \end{equation}
 
 When $r$ is a prime,
 one can combine  the equality $B_r=(N_r-N_1)/r$
 with the Weil bounds (\ref{Weil_Bounds}) to obtain the following Lemma.
 \begin{lemma}\label{lemma:closed_point_existence_p_prime}
 Let $X$ be an absolutely irreducible smooth projective algebraic curve
of genus $g$ defined over ${\mathbb F}_q$ and let $r$ be a prime number.
If $g<\frac{
\sqrt{q^r}-\sqrt{q}
}
{2}$ 
then $B_r(X)>0$.
\end{lemma}

More broadly, for any integer $r \geq 1$,
the M\"obius inversion formula, and the Weil bounds lead to the following inequality  (see for instance Proposition 3.2.10 in \cite{TVN})
\begin{equation}\label{Inequality_TVN}
\left| B_r(X)-\frac{q^r}{r} \right| \leq \Bigl(\frac{
q}{q-1}+2g\frac{\sqrt{q}}{\sqrt{q}-1}\Bigr)\frac{q^{r/2}-1}{r}<(2+7g)\frac{q^{r/2}}{r}.
\end{equation}
The asymptotic expansion $B_r=q^r/r+O(\sqrt{q^r}/r)$ follows, 
for a fixed genus $g$ and for large values of $q$.
As a second consequence of (\ref{Inequality_TVN}), we can deduce that if $g=0$ 
then $B_r(X) > 0$ 
for any $r\geq 1$.
Therefore for any $q$ we have
$$\rho(q,0)=1.$$
 The case of genus zero curves is thus solved and we will consider curves of positive genus from now on.

As a  third consequence of (\ref{Inequality_TVN}),
one can deduce the Corollary 3.2.11 in \cite{TVN}
which reads 
(where $\left\lceil x \right\rceil$ stands for the smallest integer greater than or equal to a real number $x$)
 \begin{equation}\label{Bound_TVN}
\rho(q,g)\leq \left\lceil 2\log_q\Bigl(\frac{2g+1}{\sqrt{q}-1}\Bigr)+1\right\rceil
\end{equation}
and from which one can deduce the uniform bound $\rho(q,g)\leq 4g+3$.
Combining the M\"obius inversion formula and the Weil bounds is also the point of departure
of inequalities obtained by 
Elkies et al. in \cite{Elkies}.
For instance, Lemma 2.1 in their paper states for every $r>0$ the lower bound
\begin{equation}\label{equation:Elkies_Br}
B_r(X)>\frac{{q^r-(6g+3)q^{r/2}} } {r}
\end{equation}
and provides the following uniform inequality which applies for $g \geq 2 $ 
\begin{equation}\label{equation:Elkies_rho}
\rho(q,g)\leq 2g+1.
\end{equation}

When $r \geq 2$
the first author with Haloui and Lachaud manage to improve
Bound (\ref{equation:Elkies_Br}).
In Proposition 3.7. of \cite{AHL} they study similar bounds in the context
of abelian varieties, but
the proof adapts {\sl mutatis mutandis} 
if we consider a curve rather than an abelian variety.
More precisely, they 
start from the same M\"obius inversion formula
${\displaystyle B_r=\frac{1}{r} \sum_{d\mid r} \mu(r/d)  N_d}$
where $\mu$ is the M\"obius function
and thus replace the numbers $N_d$ by their expression 
in function of the roots $\omega_j$'s. 
They prove
$rB_r 
 > G(q^{r/4})$
 where 
$G(x)=(x+1)^2 \left( (x-1)^2-2g \right)$
and so the second point of Proposition 3.7 in  \cite{AHL} remains true when the inequality becomes strict, that is
\begin{equation}\label{Lower_Bound_AHL}
B_r> \frac{(q^{r/4}+1)^2\Big((q^{r/4}-1)^2-2g \Big)}{r}.
\end{equation}

A consequence is that if $g \geq 1$, $r \geq 2$ and $(q^{r/4}-1)^2\geq 2g$ then $B_r >0$.
We have thus checked that point (i) of Proposition 3.8 in \cite{AHL}
is still valid even if $g \geq 1$ and we can state:

\begin{proposition} 
{(Aubry, Haloui and Lachaud, Proposition 3.8  in \cite{AHL})} 
\label{theorem_general}

 If $g\geq 1$ then  
\begin{equation}\label{Aubry-Haloui-Lachaud}
\rho(q,g) \leq   \Max \left( 2 , \left\lceil 4\log_q(1+\sqrt{2g}) \right\rceil \right). 
\end{equation}
\end{proposition}

One can also notice that Bound (\ref{Inequality_TVN}) on $B_r$ 
involves a second degree polynomial in $q^{r/2}$
whose study leads
to  a quite efficient bound on $\rho(q,g)$.

\begin{proposition} \label{Borne_avec_A}
If $g\geq 2$ 
or if $g=1$ and
 $q\leq 9$, then:
\begin{equation}\label{Bound:general_1}
\rho(q,g)\leq \left\lfloor 2\log_q\Bigl(\frac{A+\sqrt{A(A-4)}}{2}\Bigr) +1 \right\rfloor
\end{equation}
 where
$A:=\frac{q}{q-1}+2g\frac{\sqrt{q}}{\sqrt{q}-1}$ and $\lfloor x \rfloor$ stands for the integer part of a real number $x$.

In particular, with the same conditions on $g$ and $q$, if $g<\frac{\sqrt{q}}{2}\Bigl(1-\frac{\sqrt{q}-1}{q-1}\Bigr)$ then $\rho(q,g)=1$.

\end{proposition}

\begin{proof}
When setting $X= q^{r/2}$, the lower bound of  Inequality (\ref{Inequality_TVN}) yields
$rB_r \geq X^2-AX+A$ where $A$ is as in the statement.
The discriminant of the rigth hand side is $A(A-4)$ which is positive
if $g \geq 2$, or if $g = 1$ and $q \geq 9$.
In these cases
we see that if
$r > 2 \log_q \left( \frac{A+\sqrt{A(A-4)}}{2} \right)$ then 
$B_r\geq 1$
and the bound follows.

An easy computation shows that $2\log_q\Bigl(\frac{A+\sqrt{A(A-4)}}{2}\Bigr)<1$ if and only if 
$A< \frac{q}{\sqrt{q}-1}$,
which is equivalent to 
$g<\frac{\sqrt{q}}{2}\Bigl(1-\frac{\sqrt{q}-1}{q-1}\Bigr)$.
Hence the previous condition on $g$ implies by Inequality (\ref{Bound:general_1}) that $\rho(q,g)=1$.

\end{proof}
Now one can deduce
bounds on $\rho(q,g)$
for small values of $g$ and large values of $q$
as stated in the following corollary.
These bounds will prove useful to reduce
the work of Sections 4, 5 and 6
to a finite number of cases.

\begin{corollary}\label{corollary_q_large}
We have $\rho(q,1)=1$ for any $q\geq 8$,
$\rho(q,2)\leq 2$ for any $q\geq 7$ and
$\rho(q,3)\leq 2$ for any $q\geq 37$.
\end{corollary}
\begin{proof}
If $q \in \{8,9 \}$ then Proposition \ref{Borne_avec_A} asserts that $\rho(q,1)=1$.
Moreover, if  $q\geq 11$ then 
$A(A-4)<0$ and thus
$B_r\geq 1$ for any $r$, which implies that $\rho(q,1)=1$.
Finally, the proof of the
cases of genus 2 and 3 is straightforward by
Proposition \ref{Borne_avec_A}.
\end{proof}
To conclude this section one should also mention
that the St\" ohr-Voloch theory
leads
to specific upper and lower bounds on $B_r$ 
for irreducible plane curves. See for instance Theorem 9.62 and Theorem 9.63 in \cite{HKT}.

%
%
%

\section{Upper bound for $B_2$}

In this section we  
 focus  on the number of closed points of degree 2 
 for which we  improve 
the known upper bounds.

We begin by summarizing the upper bounds obtained in the most natural ways.
The first idea is to specialize    Inequality (\ref{Inequality_TVN}) for $r=2$ to obtain
$$B_2\leq \frac{q^2+1+2gq}{2} +\frac{q-1+2g\sqrt{q}}{2}.$$
For comparisons it is interesting to adopt 
the asymptotic viewpoint for a fixed genus $g$
and for large values of $q$.
This way the asymptotic expansion of the upper bound reads
\begin{equation}\label{B_2_asymptotic_TVN}
B_2 \leq \frac{q^2}{2} +(g+\frac{1}{2})q +O(\sqrt{q}).
\end{equation}

Another natural idea is to start from 
 Formula (\ref{link_N_r-B_d}), namely
$N_2=B_1+2B_2$. Then the Weil upper bound for $N_2$ yields to  the following bound which involves $N_1$:
\begin{equation}\label{B_2_Weil_avec_N_1}
B_2=\frac{N_2-N_1}{2}\leq \frac{q^2+1+2gq}{2} -\frac{N_1}{2}.
\end{equation}

Again from the Weil lower bound for $N_1$, 
 one deduces the following bound which only depends on $q$:
\begin{equation}\label{B_2_Weil}
B_2 \leq \left\{
    \begin{array}{ll}
        \frac{q^2+1+2gq}{2} & \mbox{if} \ \ g \geq \frac{q+1}{2\sqrt{q}} \\
         \frac{q^2-q+2g(q+\sqrt{q})}{2} & \mbox{otherwise.}
    \end{array}
\right.
\end{equation}
If we adopt the same asymptotic point of view for large values of $q$ we get
\begin{equation}\label{B_2_asymptotic_Weil}
B_2 \leq \frac{q^2}{2} +(g-\frac{1}{2})q +O(\sqrt{q}).
\end{equation}

Let us now present our new upper bound.
\begin{proposition}\label{Borne_B2}
Let $X$ be a curve of genus $g>0$ defined over ${\mathbb F}_q$. 

 We have:
\begin{equation}\label{equation:B2}
B_2(X) \leq \left\{
    \begin{array}{ll}
        \frac{q^2+1+2gq}{2}-\frac{(q+1)^2}{2g} & \mbox{if} \ \ g \geq 2q+2 \\
         \frac{q^2+1+2gq}{2}-\frac{4(q+1)-g}{8} & \mbox{otherwise.}
    \end{array}
\right.
\end{equation}
\end{proposition}

The following recent result is the point of departure to prove 
Proposition \ref{Borne_B2}. It has been 
interpreted by Hallouin and Perret 
as a consequence of an inequality of Euclidean geometry
in the numerical space $Num(X \times X)_{\mathbb{R}}$.
\begin{thm} {(Hallouin and Perret, Proposition 12 in  \cite{Hallouin-Perret})}
Let $X$ be a curve of genus $g$ defined over $\mathbb{F}_q$. Then
\begin{equation}
\sharp X({\mathbb F}_{q^2}) - (q^2+1) \leq 2gq-\frac{1}{g}\bigl(\sharp X({\mathbb F}_{q}) - (q+1)\bigr)^2.
\end{equation}
\end{thm}

\begin{proof}
The previous theorem gives immediately 
 (see also Proposition 3.1. in \cite{A-I} where it has already been noticed):
\begin{equation}\label{Halloin-Perret_Aubry-Iezzi}
B_2\leq  \frac{q^2+1+2gq}{2}-\frac{N_1}{2}-\frac{(N_1-(q+1))^2}{2g}
\end{equation}
which clearly improves Bound (\ref{B_2_Weil_avec_N_1}).
Now we consider the function $x\mapsto -\frac{1}{2g}(x-(q+1))^2-\frac{x}{2}$. 
If $g\geq 2q+2$  this function reaches its maximum 
for $x=0$ which implies  $B_2\leq  \frac{q^2+1+2gq}{2}-\frac{(q+1)^2}{2g}$. 
If $g\leq 2q+2$ it reaches its maximum for $x=\frac{2q+2-g}{2}$ and thus we deduce $B_2\leq  \frac{q^2+1+2gq}{2}-\frac{4(q+1)-g}{8}$.
\end{proof}

Straightforward computations in the three
intervals $[1, \frac{q+1}{2\sqrt{q}}]$, $[ \frac{q+1}{2\sqrt{q}}, 2q+2]$, and $[2q+2, \infty)$
enable us to check that the new bounds of Proposition \ref{Borne_B2} 
are always better than  Bound (\ref{B_2_Weil}) deduced from the Weil bounds.

Let us come back to the asymptotic viewpoint for
a fixed genus $g$ and for large values of $q$ to compare
 Bound (\ref{B_2_Weil})
with our  new bound (\ref{equation:B2}) which reads
\begin{equation}\label{B_2_asymptotic_Weil}
B_2 \leq \frac{q^2}{2} +(g-\frac{1}{2})q +O(1).
\end{equation}
In the asymptotic expansion of the upper bound 
we have managed to replace a $O(\sqrt{q})$ by a $O(1)$ .

We now provide examples to show that 
the integer part of the new bound 
(\ref{equation:B2}) of
Proposition \ref{Borne_B2} is reached 
for different values of $g$ and $q$.
The third and fourth columns of the following table
enable us to compare the bound (\ref{B_2_Weil})
deduced from the Weil bounds
to our new bound (\ref{equation:B2}) for some values $q$ and $g$.
The fifth column provides  equations of curves which attain
 our new bound. 
 To describe the curves
we sometimes need to introduce  
a generator $a$ of the multiplicative group ${\mathbb F}_q^{\ast}$.
We also give the numbers of rational points $N_1$ and $N_2$ 
over ${\mathbb F}_q$ and ${\mathbb F}_{q^2}$ of the given curve.

\begin{center}
\begin{table}[!h]
\centering
\begin{tabular}{|c|c|c|c|l|c|c|}
\hline
 $g$& $q$&Bound    & Bound    &  Curves reaching bound
 (\ref{equation:B2})
  of Proposition (\ref{Borne_B2})  & $N_1$ & $N_2$ \\
 & & (\ref{B_2_Weil})    & (\ref{equation:B2})  & & &  \\
\hline
$1$ & 2  & 4 & 3 &  $y^2+xy=x^3+x^2+1$   & $2$ & $8$  \hfill \\
\hline
1& 3 & 7 & 6 &  $y^2=x^3+x^2-1$ & 3 &15 \hfill  \\
\hline
1 &4 & 12 & 10 & $y^2+xy=x^3+a$  & 4 &24 \hfill \\
\hline
2 &3& 11 & 9 &  $y^2=2x^6+x^4+2x^3+x^2+2$ & 2 & 20  \hfill\\ 
\hline
2&4& 16 & 14 &   $y^2+(x^2+x)y=x^5+x^3+x^2+x$  & 3 & 31  \hfill \\
\hline
2&5 & 23 & 20 &   $y^2=4x^6+x^5+x^4+x^3+x^2+x+4$  & 4 & 44 \hfill \\
\hline
2 & 7 & 39 & 35 &   $y^2=3x^6+3x^3+3$  & 6 & 76  \hfill \\
\hline
2 &8 & 48 & 44 &   $y^2+(x^2+x)y=a^2x^5+a^2x^3+ax^2+ax$ & 7 & 95 \hfill \\
\hline
2 &9 & 59 & 54 &  $y^2=2ax^6+ax^5+2ax^4+ax^2+ax+a$ & 8 & 116 \hfill \\
\hline
$2$ &11 & 83 & 77 &  $y^2=7x^6+5x^5+9x^4+8x^3+5x^2+6x+7$ & 10 & 164 \hfill \\
\hline
$3$ & $2$ & 8 & 7  & $x^4+x^2y^2+x^2yz+x^2z^2+xy^2z+xyz^2+y^4+y^2z^2+z^4=0$ & 0 & 14\\
\hline
\end{tabular}
\caption{Ex. of curves which attain the bound of Proposition \ref{Borne_B2}.\\ 
In column $5$ we denote by $a$ a generator of $\mathbb{F}_q^{\ast}$}
\end{table}
\end{center}


\section{Elliptic curves}
In this section we consider an elliptic curve $X$  defined over ${\mathbb F}_q$.

\begin{theorem} \label{theorem_genre_1} 
The values of $\rho(q,1)$ are as follows.\\ \ \\
\begin{tabular}{ccccl}
        (i) & $\rho(2,1)=5$ & &  (ii) &  $\rho(3,1)=3$  \\ \\
       
         (iii) & $\rho(4,1) = 3$ & & (iv) & $\rho(q,1)=1$ for any $q \geq 5$ \\  \\

\end{tabular} 

\end{theorem}

\begin{proof}
(i) First we consider the case where $q=2$.
Proposition \ref{theorem_general} leads to $\rho(2,1)\leq 6$.
 Lemma \ref{lemma:closed_point_existence_p_prime}
also applies and ensures the existence of a degree $5$ point,
so  $\rho(2,1)\leq 5$.
Note that a curve $X$  defined over ${\mathbb F}_2$ 
 satisfies
$B_4(X)=0$
 if and only if
it is a DS-curve  for the extension 
${\mathbb F}_{2^4}/{\mathbb F}_{2^2}$.
As Lario indicates in \cite{Lario} the elliptic curve of equation $y^2+y=x^3$ 
is an example of such a curve and we can conclude that $\rho(2,1)= 5$.

(ii) For $q=3$ we use Proposition \ref{Borne_avec_A} to get $\rho(3,1)\leq 3$.
Lario gives in \cite{Lario}
a DS-curve of genus 1 for the extension ${\mathbb F}_9/{\mathbb F}_3$, namely the curve of equation $y^2=x^3+2x+1$.
So there exists an elliptic curve defined over ${\mathbb F}_3$ with no closed point of degree 2, and thus we have that $\rho(3,1)\geq 3$ which gives the result.

(iii) For $q=4$ Proposition \ref{Borne_avec_A} gives $\rho(4,1)\leq 3$ and 
as the elliptic curve of equation $y^2+y=x^3$
defined  in \cite{Lario} has no closed point of degree 2 we can conclude that   
 $\rho(4,1)= 3$.

(iv) We use Proposition \ref{Borne_avec_A} 
to get $\rho(5,1)\leq 2$ and $\rho(7,1)\leq 2$.
 But 
any elliptic curve 
over a finite field has a rational point,
so  $\rho(5,1)= 1$ and $\rho(7,1)= 1$.
At last, Corollary \ref{corollary_q_large} shows that for any $q\geq 8$ we have $\rho(q,1)= 1$, which concludes the proof.
\end{proof}





\section{Genus 2 curves}\label{section:genus_2}

In this section we consider a curve $X$ of genus 2 defined over ${\mathbb F}_q$.
To make the proof of Theorem \ref{theorem_genre_2} more readable
we first establish the following lemma.
\begin{lemma}\label{closed_point_genre_2}
Any absolutely irreducible smooth projective algebraic curve $X$ of genus 2 
 defined over ${\mathbb F}_2$
 (respectively over ${\mathbb F}_4$)  admits at least one closed point of degree $4$
 (respectively of degree $2$ and $3$).
 \end{lemma}
\begin{proof}
We can relate the issue to the existence of a DS-curve
 and thus conclude with \cite{Lario}.
 \end{proof}
As an important ingredient of the following theorem we will also make use of the classification of genus $2$ curves
over $\mathbb{F}_q$ up to $\mathbb{F}_q$-isomorphism and quadratic twist provided by Maisner and Nart   in \cite{MNH}.

\begin{theorem} \label{theorem_genre_2}

The values of $\rho(q,2)$ are as follows. \\

\begin{tabular}{cccccccc}
(i) &   $\rho(2,2)=4$ & (ii) & $\rho(4,2)=2$ & \ \ \ & (v) & $\rho(q,2)=2$ for $7 \leq q \leq 11$ \\ \ \\
(iii) &  $\rho(3,2) = 4$ & (iv) & $\rho(5,2)=3$ & \ \ \ & (vi) & $\rho(q,2)=1$ \textrm{for any} $q \geq 13$
\end{tabular} 

\end{theorem}

\begin{proof}
First,
Corollary \ref{corollary_q_large}
 implies that for $q\geq 7$ we have 
$\rho(q,2)\leq 2$. 
On the other hand Theorem 3.2. in \cite{MNH} ensures 
that there is no genus $2$ pointless curve defined over $\mathbb{F}_q$ if $q>11$.
Hence we conclude that $\rho(q,2)= 1$ for any $q\geq 13$ and point $(vi)$ follows.

The same reference provides an example 
of a pointless genus 2 curve defined over ${\mathbb F}_q$
for any $q\leq 11$.
It implies that $\rho(q,2)\geq 2$ for $q\leq 11$, and 
thus point $(v)$ is proved.

Now we consider the case where $q=2$.
The genus 2 curve $X$  
described
in Table 2 of \cite{MNH} 
 by the equation $y^2+y=x^5+x^2$ 
satisfies
$N_1(X)=N_3(X)=5$, which means that
 $B_3(X)=0$, and thus $\rho(2,2)\geq 4$.
 Lemma (\ref{closed_point_genre_2})
ensures the existence of a closed point of degree $4$ on any curve
whereas (\ref{equation:Elkies_rho}) gives $\rho(2,2)\leq 5$.
We can conclude that $\rho(2,2)=4$.

When $q=3$, Proposition \ref{theorem_general} asserts that $\rho(3,2)\leq 4$.
 The genus 2 curve $X$
defined in Table 3 in \cite{MNH}  and also in \cite{Lario}
 by the equation $y^2=(1+x^2)(-1+x+x^2)(-1-x+x^2)$
is such that 
$N_3(X)=N_1(X)=8$, and so
 $B_3(X)=0$ which implies $\rho(3,2)=4$.

 If $q=4$ we use Proposition \ref{theorem_general} to get $\rho(4,2)\leq 4$ 
 and Lemma \ref{closed_point_genre_2} to conclude.

Now suppose that $q=5$.  Proposition \ref{theorem_general} gives $\rho(5,2)\leq 3$ but
the genus 2 curve $X$ defined   
in \cite{Lario}
by the equation $y^2=x^5+4x$
satisfies  $B_2(X)=0$,
and the other inequality follows.
 
\end{proof}





\section{Genus 3 curves}\label{section:genus_3}

We would like to emphaze two of the main ingredients of the proof of the main theorem of this section.
First we will make use of the
{\sl L-functions and modular forms database} 
(sometimes refered as LMFDB, see \cite{LMFDB}).
Second, we will exploit the
Theorem 1.1. proved by Howe, Lauter and Top in \cite{HLT} 
which ensures 
that there exists a pointless genus 3 curve defined over ${\mathbb F}_q$
if and only if either $q=32$, $q=29$ or $q\leq25$. 
Hence for such values of $q$
we know that $\rho(q,3)\geq 2$ .

\begin{theorem} \label{theorem_genre_3}

The values of $\rho(q,3)$ are as follows. \\

\begin{center}
\begin{tabular}{cccccccl}
(i)    & $\rho(2,3)=7$    & (ii)  &  $\rho(3,3) = 5$ & (iii) &   $\rho(4,3) =3$  &   (iv)  & $\rho(5,3)=3$  \\ \\
 (v)  &  $\rho(7,3)=2$   & (vi) &  $\rho(8,3)=2$   & (vii) &  $\rho(9,3)=3$   &  (viii)  & $\rho(q,3)=2$ for $11 \leq q \leq 25$ \\  \\
 (ix)  & $\rho(27,3)=1$ & (x)  &  $\rho(29,3)=2$ & (xi)  &  $\rho(31,3)=1$ & (xii) & $\rho(32,3)=2$  \\ \\
& &   &  &   &  & (xiii)   &  $\rho(q,3)=1$ for any $q \geq 37$  \\ \\
       
\end{tabular} 
\end{center}

\end{theorem}

\begin{proof}

As already mentioned, Theorem 1.1 of \cite{HLT} implies  that $\rho(q,3)\geq 2$ for $q\leq 25$, $q=29$ and $q=32$.

(i) 
The inequality (\ref{equation:Elkies_rho}) yields
 $\rho(2,3)\leq 7$.
Moreover,  LMFDB provides the curve $X$ 
of equation 
$x^4+x^2y^2+x^2yz+x^2z^2+xy^2z+xyz^2+y^4+y^2z^2+z^4=0$
which satisfies $B_6(X)=0$, and thus $\rho(2,3)=7$.


(ii) By Proposition \ref{theorem_general}  we have 
$\rho(3,3)\leq 5$.
But the curve $x^3z+xz^3+y^4=0$ given in  \cite{LMFDB} is a curve of genus 3 over ${\mathbb F}_3$  without points of degree 4, so $\rho(3,3)> 4$ and 
we are done.

(iii) Proposition \ref{theorem_general}   gives
$\rho(4,3)\leq 4$
whereas the curve $x^3y+x^3z+x^2y^2+xz^3+y^3z+y^2z^2=0$ defined over ${\mathbb F}_4$ and provided by \cite{LMFDB}   
has genus 3 and no points of degree 2. So $\rho(4,3)\geq 3$.
Furthermore, the Weil bounds  (\ref{Weil_Bounds}) yield
$N_1\leq 17 \leq N_3$.
But, there does not exist a curve over ${\mathbb F}_4$ of genus 3 with 17 rational points since $N_4(3)=14$
 (see Table 1 in \cite{Serre}).
Now  use
$B_3=(N_3-N_1)/3$ to get $B_3>0$ and the desired result.

(iv) Here Proposition \ref{Borne_avec_A}  improves 
Proposition (\ref{theorem_general}) and we get 
$\rho(5,3) \leq 3$.
Moreover, the curve $y^2=x^7+x^5+3x^3+x$ found in \cite{LMFDB} has no points of degree 2,
so $\rho(5,3) \geq 3$.

(v) Proposition \ref{theorem_general}  gives $\rho(7,3) \leq 3$.
Moreover, by Theorem 1 of \cite{V},  there does not exist a DS-curve  for the extension ${\mathbb F}_{7^2}  / {\mathbb F}_7$.
So one can conclude that $\rho(7,3)=2$.

(vi) Proposition \ref{theorem_general}  leads to  $\rho(8,3) \leq 3$. 
And by \cite{LMFDB} there is no genus $3$ curve $X$ 
defined over $\mathbb{F}_8$ such that $B_2(X)=0$.
One can see the coherence with results provided in \cite{Lario}: 
Lario states that there does not exist a DS-curves 
for the field extension $\mathbb{F}_{64} / \mathbb{F}_8$. 

(vii) By Proposition \ref{theorem_general}  again we have $\rho(9,3) \leq 3$.
Moreover, the genus 3 curve $X$ of equation $x^4+y^4+z^4=0$ defined over ${\mathbb F}_9$ 
proposed in
\cite{Lario}
has no degree 2 points since $N_1(X)=N_2(X)=28$.
This implies $\rho(9,3)=3$. 

(viii) 
We 
use the bound of Proposition (\ref{theorem_general})
to find
$\rho(q,3) \leq 2$
for $11 \leq q \leq 25$.
But for such values of $q$  Theorem 1.1. of \cite{HLT} states that there exists a pointless curve of genus $3$.

%
%
%
%
%

(ix), (x), (xi) and (xii)  We first use Proposition \ref{theorem_general} to learn that  $\rho(q,3) \leq 2$ for $q \in \{ 27, 29, 31, 32\}$.
Moreover, by Theorem 1.1. of \cite{HLT}, 
there exists a pointless curve defined over $\mathbb{F}_{29}$ nor $\mathbb{F}_{32}$
whereas there does not exist such a curve neither over $\mathbb{F}_{27}$ nor $\mathbb{F}_{31}$.

(xiii) 
Corollary \ref{corollary_q_large} enables us to 
check that for any $q \geq 37$ we have 
 $\rho(q,3) \leq 2$.
But  by Theorem 1.1. of \cite{HLT}, there does not exist a genus $3$ pointless curve for such values of $q$, so
$\rho(q,3)=1$.
\end{proof}

{\bf Acknowledgement.} The authors would like to thank 
Joan-C. Lario
for a useful discussion on  Diophantine  
Stability.

\bigskip
\bibliographystyle{plain}

\begin{thebibliography}{1}

\bibitem{AHL}
Y. Aubry, S. Haloui and G. Lachaud.
\newblock On the number of points on abelian and Jacobian varieties over finite fields,
\newblock Acta Arith. 160.3  (2013), 201--242.


\bibitem{A-I}
Y. Aubry and A. Iezzi.
\newblock Optimal and maximal singular curves,
\newblock Contemp. Math. Vol. 686, pp 31--43, A.M.S.  (2017).


\bibitem{Elkies}
N. D. Elkies, E. W. Howe, A. Kresch, B. Poonen, J. L. Wetherell and M. E. Zieve.
\newblock Curves of every genus with many points. II. Asymptotically good families,
\newblock Duke Math. J. 122 (2004), 399--422.

\bibitem{Hallouin-Perret}
E. Hallouin and M. Perret.
\newblock A unified viewpoint for upper bounds for the number of points of curves over finite fields via euclidean geometry and semi-definite symmetric Toeplitz matrices,
\newblock Trans. A.M.S., Volume 372, Number 8, 15 October 2019, pp 5409--5451.


\bibitem{HKT}
J.W.P. Hirschfeld, G. Korchm\' aros  and F. Torres.
\newblock Algebraic curves over a finite field,
\newblock Princeton Ser. Appl. Math.
\newblock Princeton University Press, Princeton, NJ, 2008. 




\bibitem{HLT}
E. Howe, K. Lauter and J. Top.
\newblock Pointless curves of genus three and four,
\newblock Arithmetic, geometry and coding theory (AGCT 2003), 125--141, S\'emin. Congr., 11, Soc. Math. France, Paris, 2005.


\bibitem{Ihara}
Y. Ihara.
\newblock Some remarks on the number of rational points of algebraic curves over finite fields,
\newblock J. Fac. Sci. Univ. Tokyo Sect. IA Math., 28(3) : 721--724 (1982), 1981.


\bibitem{Lario}
J-C. Lario.
\newblock Curves over finite fields with Diophantine stability,
\newblock \url{https://web.mat.upc.edu/joan.carles.lario/DS.html#}


\bibitem{LMFDB}
The LMFDB Collab., The $L$-functions and modular forms database, http://www.lmfdb.org (2022).


\bibitem{MNH}
D. Maisner and E. Nart , with an Appendix by E. Howe.
\newblock Abelian surfaces over finite fields as Jacobians,
\newblock Experiment. Math. 11 (2002), no.3, 321--337.

\bibitem{M-R}
B. Mazur and K. Rubin , with an Appendix by M. Larsen.
\newblock Diophantine stability,
\newblock American J. Math. Vol. 140, Number 3 (2018), pp. 571--616.

\bibitem{Serre}
J.- P. Serre. 
\newblock Rational points on curves over finite fields,
\newblock Doc. Math. (Paris), vol. 18 Soc. Math. France (2020).


\bibitem{TVN}
M. Tsfasman, S. Vl\u adu\c t  and D. Nogin.
\newblock Algebraic geometric codes: basic notions,
\newblock Mathematical Surveys and Monographs, 139.
\newblock American Mathematical Society, Providence, RI, 2007.



\bibitem{V}
B. Vrioni.
\newblock A census for curves and surfaces with Diophantine Stability  over finite fields,
\newblock Ph.D. thesis from Polytechnic University of Catalonia, School of Mathematics and Statistics, 2021.




\bibitem{D-W}
W. C. Waterhouse.
\newblock Abelian varieties over finite fields,
\newblock Ann. Sc. E. N. S. (4), {\bf 2} (1969), 521--560.





\end{thebibliography}

\end{document}